\documentclass[a4paper, 11pt,]{article}

\usepackage{amsmath,amsthm}
\usepackage{amsfonts,amssymb}
\usepackage{graphicx}
\usepackage{calc}
\usepackage{multicol}
\def\normo#1{\left\|#1\right\|}
\def\abs#1{|#1|}
\def\aabs#1{\big|#1\big|}
\def\brk#1{\left(#1\right)}
\def\fbrk#1{\{#1\}}
\def\rev#1{\frac{1}{#1}}

\newtheorem{theorem}{Theorem}[section]

\newtheorem{proposition}[theorem]{Proposition}
\newtheorem{definition}[theorem]{Definition}
\newtheorem{lemma}[theorem]{Lemma}
\newtheorem{corollary}[theorem]{Corollary}

\newtheorem{remark}[theorem]{Remark}

\begin{document}
\title{{ $L^p$ Boundedness of Commutators of Riesz Transforms \\Associated to Schr\"{o}dinger Operator}\thanks{
Research supported by NNSF of China No.10471002, RFDP of China No:
20060001010.
 } }

\author{Zihua Guo, Pengtao Li, and Lizhong Peng  \date{}\\
{\small
LMAM, School of Mathematical Sciences, Peking University}\\ {\small Beijing 100871, China}\\
{\small E-mail: zihuaguo@math.pku.edu.cn, li\_ptao@163.com,
lzpeng@pku.edu.cn} }\maketitle

%\begin{minipage}{4.9in}

%\begin{multicols}{2}
{\bf Abstract:}\quad In this paper we consider $L^p$ boundedness of
some commutators of Riesz transforms associated to Schr\"{o}dinger
operator $P=-\Delta+V(x)$ on $\mathbb{R}^n,\ n\geq 3$. We assume
that $V(x)$ is non-zero, nonnegative, and belongs to $B_q$ for some
$q \geq n/2$. Let $T_1=(-\Delta+V)^{-1}V,\
T_2=(-\Delta+V)^{-1/2}V^{1/2}$ and $T_3=(-\Delta+V)^{-1/2}\nabla$.
We obtain that $[b,T_j]\ (j=1,2,3)$ are bounded operators on
$L^p(\mathbb{R}^n)$ when $p$ ranges in a interval, where $b \in
\mathbf{BMO}(\mathbb{R}^n)$. Note that the kernel of $T_j\
(j=1,2,3)$ has no smoothness.


{\bf 2000 MS Classification:} 47B38, \ 42B25, \ 35Q40

\section{Introduction}
Let $P=-\Delta+V(x)$ be the Schr\"odinger differential operator on
$\mathbb{R}^n,n\geq3$. Throughout the paper we will assume that
$V(x)$ is a non-zero, nonnegative potential, and belongs to $B_q$
for some $q > n/2$. Let $T_{j}\ (j=1,2,3)$ be the Riesz transforms
associated to Schr\"{o}dinger operators, namely,
$T_1=(-\Delta+V)^{-1}V,\ T_2=(-\Delta+V)^{-1/2}V^{1/2}$ and
$T_3=(-\Delta+V)^{-1/2}\nabla$. $L^p$ boundedness of $T_j\
 (j=1,2,3)$ was widely studied(\cite{SH}, \cite{ZH}). In this paper, we will
discuss the $L^p$ boundedness of the commutator operators
$[b,T_j]=bT_j-T_jb\ (j=1,2,3)$, where $b \in
\mathbf{BMO}(\mathbb{R}^n)$.

A nonnegative locally $L^{q}$ integrable function $V(x)$ on
$\mathbb{R}^{n}$ is said to belong to $B_{q}\ (1<q<\infty)$, if
there exists $C>0$ such that the reverse H\"{o}lder inequality
\begin{eqnarray}\label{e1}
\brk{\frac{1}{|B|}\int_{B}V^{q}dx}^{\frac{1}{q}}\leq
C\brk{\frac{1}{|B|}\int_{B}Vdx}
\end{eqnarray}
holds for every ball $B$ in $\mathbb{R}^{n}$.

\begin{remark}\label{r1} By H\"older inequality we can get that
$B_{q_1}\subset B_{q_2}$, for $q_1\geq q_2>1$. One remarkable
feature about the $B_q$ class is that, if $V \in B_q$ for some
$q>1$, then there exists $\epsilon >0$, which depends only on n and
the constant C in (\ref{e1}), such that $V \in B_{q+\epsilon}$
(\cite{GE}). It's also well known that, if $V\in B_q,\ q>1$, then
$V(x)dx$ is a doubling measure, namely for any $r>0,\ x\in
\mathbb{R}^n$,
\begin{eqnarray}\label{e2}
\int_{B(x,2r)}V(y)dy\leq C_0 \int_{B(x,r)}V(y)dy.
\end{eqnarray}
\end{remark}

It was proved that if $V \in B_n$, then $T_3$ is a
Calder\'{o}n-Zygmund operator (\cite{SH}). According to the
classical result of R.Coifman, R.Rochberg, and G.Weiss (\cite{CRW}),
$[b,T_3]$ is bounded on $L^p\ (1<p<\infty)$ in this case. So we
restrict ourselves to the case that $V \in B_q \ (n/2 < q <n)$, when
considering $[b,T_3]$.

We recall that an operator $T$ taking $C_c^{\infty}(R^n)$ into
$L_{loc}^1(R^n)$ is called a Calder\'{o}n-Zygmund operator if

(a) $T$ extends to a bounded linear operator on $L^2(R^n)$,

(b) there exists a kernel K such that for every $f \in
L_c^{\infty}(R^n)$, $$Tf(x)=\int_{R^n}K(x,y)f(y)dy\quad a.e.\ on\
\{suppf\}^c,$$

(c) the kernel $K(x,y)$ satisfies the Calder\'{o}n-Zygmund estimate
\begin{eqnarray}
&|K(x,y)| \leq \frac{C}{|x-y|^{n}};\label{e3}\\
&|K(x+h,y)-K(x,y)|\leq \frac{C|h|^{\delta}}{|x-y|^{n+\delta}};\label{e4}\\
&|K(x,y+h)-K(x,y)|\leq
\frac{C|h|^{\delta}}{|x-y|^{n+\delta}};\label{e5}
\end{eqnarray}
for $x,y\in R^{n},|h|<\frac{|x-y|}{2}$ and for some $\delta >0$.

If $T$ is a Calder\'{o}n-Zygmund operator, $b\in \mathbf{BMO}$, the
 boundedness on every $L^p\ (1<p<\infty)$ of [b,T] was first discovered
 by Coifman, Rochberg and Weiss (\cite{CRW}). Later, Str\"{o}mberg \cite{JA} gave a simple
proof, adopting the idea of relating commutators with the sharp
maximal operator of Fefferman and Stein. In both proof, the
smoothness of the kernel (\ref{e4}) plays a key role. However, in
our problem the kernel has no smoothness of this kind due to $V$.
This difficulty can be overcome by our basic idea. We discover that
the kernels have some other kind of smoothness.
\quad\begin{definition} K(x,y) is said to satisfy $H(m)$ for some
$m\geq 1$, if there exists a constant $C>0$, such that, $\forall\
l>0, \ x, x_0 \in \mathbb{R}^n$ with $\abs{x-x_0} \leq l $, then
\begin{eqnarray}\label{e6}
\sum_{k=5}^{\infty} k(2^kl)^{\frac{n}{m'}}\brk{\int_{2^kl \leq
\abs{y-x_0} < 2^{k+1}l} \abs{K(x,y)-K(x_0,y)}^mdy}^{1/m}<C,
\end{eqnarray}
where $1/m'=1-1/m$.
\end{definition}

This kind of smoothness was not new. We find that the case $m=1$ was
given by Meyer(\cite{ME}). It's easily seen that if $K(x,y)$
satisfies (\ref{e4}), then $K(x,y)$ satisfies $H(m)$ for every
$m\geq 1$. By H\"older inequality we can get that if $K(x,y)$
satisfies $H(m)$ for some $m\geq 1$, then $K(x,y)$ satisfies $H(t)$
for $1\leq t\leq m$. We now list some results concerning $L^p$
boundedness of $T_j\ (j=1,2,3)$, and refer the readers to \cite{SH}
for further details. We will adopt the notation $1/p'=1-1/p$ for $p
\geq 1$ throughout the paper.

\begin{theorem}[Theorem 3.1, \cite{SH}]\label{tA}
Suppose $V \in B_q$ and $ q \geq n/2$. Then, for $q'\leq p \leq
\infty$,
$$\normo{(-\Delta+V)^{-1}V f}_p \leq C_{p}\|f\|_{p}.$$
\end{theorem}

\begin{theorem}[Theorem 5.10, \cite{SH}]\label{tB}
Suppose $V \in B_q$ and $ q \geq n/2$. Then, for $(2q)'\leq p \leq
\infty$,
$$\normo{(-\Delta+V)^{-1/2}V^{1/2} f}_p \leq C_{p}\|f\|_{p}.$$
\end{theorem}

\begin{theorem}[Theorem 0.5, \cite{SH}]\label{tC}
Suppose $V \in B_q$ and $\frac{n}{2}\leq q <n$. Let
$(1/{p_0})=(1/q)-(1/n)$. Then, for $p_0'\leq p< \infty$,
$$\normo{(-\Delta+V)^{-1/2}\nabla f}_{p}\leq C_{p}\|f\|_{p}.$$
\end{theorem}

The basic idea in \cite{SH} is that, to exploit a pointwise estimate
of the kernel and the comparision to the kernel of classical Riesz
transform. Generally, it is based on the following two basic facts.
If $V$ is large, then one expects the kernel itself has a good
decay. On the other hand, if $V$ is small, then it is close to the
classical Riesz transform. In this paper, we adopt a different idea.
Since we know that the kernel do not satisfy the Calder\"on-Zygmund
estimate (\ref{e4}), we study how close it is. See Section 2.

We will show that the kernels have very good smoothness with respect
to the first variable of the following strong type. It is almost
(\ref{e4}). There exists a constant $C>0$ and $\delta>0$, such that,
for some $m>1$, $\forall\ l>0, \ x, x_0 \in \mathbb{R}^n$ with
$\abs{x-x_0} \leq l $, then
\begin{eqnarray}\label{e25}
\sum_{k=5}^{\infty} 2^{k\delta}(2^kl)^{\frac{n}{m'}}\brk{\int_{2^kl
\leq \abs{y-x_0} < 2^{k+1}l} \abs{K(x,y)-K(x_0,y)}^mdy}^{1/m}<C.
\end{eqnarray}

Recall that $T_1=(-\Delta+V)^{-1}V,\ T_2=(-\Delta+V)^{-1/2}V^{1/2}$,
and $T_3=(-\Delta+V)^{-1/2}\nabla$. Now we state our main results.

\begin{theorem}\label{t1}
Suppose $V\in B_{q}$ and $q\geq n/2$. Let $b\in \mathbf{BMO}$. Then,
we have\\
(i) If $q'\leq p < \infty$,
$$\normo{[b,T_1]f}_p \leq C_p \normo{b}_{\mathbf{BMO}}
\normo{f}_p;$$ (ii) If $(2q)'\leq p < \infty$,
$$\normo{[b,T_2] f}_p \leq C_{p}
\normo{b}_{\mathbf{BMO}}\|f\|_{p};$$ (iii) If $p_0'\leq p< \infty$,
let $(1/{p_0})=(1/q)-(1/n)$,
$$\normo{[b,T_3] f}_p \leq C_{p}
\normo{b}_{\mathbf{BMO}}\|f\|_{p}.$$
\end{theorem}

We know that $T_1^*=V(-\Delta+V)^{-1},\
T_2^*=V^{1/2}(-\Delta+V)^{-1/2}$, and
$T_3^*=-\nabla(-\Delta+V)^{-1/2}$. By duality we can easily get that
$$\normo{[b,T_1^*]f}_p \leq C_p \normo{b}_{\mathbf{BMO}} \normo{f}_p,\ 1<p\leq q,$$
$$\normo{[b,T_2^*] f}_p \leq C_{p} \normo{b}_{\mathbf{BMO}}\|f\|_{p},\ 1<p\leq 2q,$$
$$\normo{[b,T_3^*] f}_p \leq C_{p} \normo{b}_{\mathbf{BMO}}\|f\|_{p},\ 1<p\leq p_0.$$

From Theorem 1 (i), we can get the result concerning second order
Riesz transform. Let $T_4=(-\Delta+V)^{-1}\nabla^2$, then
$T_4^*=\nabla^2(-\Delta+V)^{-1}$. Indeed,
$T_{4}=(-\triangle+V)^{-1}\nabla^{2}=(-\triangle+V)^{-1}\triangle\triangle^{-1}\nabla^{2}=(I-(-\triangle+V)V)\frac{\nabla^{2}}{\triangle}=(I-T_{1})\frac{\nabla^{2}}{\triangle}$.
We have
\begin{corollary}\label{t9}
Suppose $V \in B_q$ and $ q \geq n/2$. Then
$$\normo{[b,T_4] f}_p \leq C_{p} \normo{b}_{\mathbf{BMO}}\|f\|_{p},\ q'\leq p <
\infty,$$ and
$$\normo{[b,T_4^*] f}_p \leq C_{p} \normo{b}_{\mathbf{BMO}}\|f\|_{p},\ 1< p \leq q.$$
\end{corollary}

For classical Riesz transform, the converse problem was also
considered in \cite{CRW}. This implies a new characterization of
$\mathbf{BMO}$. In this paper we also discuss the converse problem.
Namely, if $[b,T_3]$ is bounded on $L^2$, do we have $b \in
\mathbf{BMO}$? The answer is negative for general $V \in B_q$. It is
due to that, for some good $V$, the kernel of $T_3$ is better than
that of Riesz transform, which makes that the commutator can absorb
mild singularity. We give a counterexample for $V\equiv1$. On the
other hand, if imposing some integrability condition on $V$, we can
have the converse.

Throughout this paper, unless otherwise indicated, we will use $C$
and $c$ to denote constants, which are not necessarily the same at
each occurrence. By $A \sim B$, we mean that there exist constants
$C>0$ and $c>0$, such that $c\leq A/B\leq C$.

The paper is orgnized as following. In Section 2, we will give the
estimates of the kernels $K_{j}(j=1,2,3)$ of the operators $T_{j}$.
The proof of Theorem 1 is stated in Section 3. In Section 4, we
discuss the converse problem.

\section{Estimate of the kernels}
This section is devoted to give the estimate of the kernels
associated to $T_j\ (j=1,2,3)$ and denoted by $K_j(x,y)\ (j=1,2,3)$
respectively. Let $\Gamma(x,y,\tau)$ denote the fundamental solution
for the Schr\"odinger operator $-\Delta+(V(x)+i\tau),\ \tau \in
\mathbb{R}$, and $\Gamma_0(x,y,\tau)$ for the operator
$-\Delta+i\tau,\ \tau \in \mathbb{R}$. Clearly,
$\Gamma(x,y,\tau)=\Gamma(y,x,-\tau)$.

For $x\in \mathbb{R}^{n}$, the function $m(x,V)$ is defined by
$$\frac{1}{m(x,V)}=sup\ \{r>0: \frac{1}{r^{n-2}}\int{B(x,r)}{}V(y)dy\leq 1\}.$$
The function $m(x,V)$ reflects the scale of $V(x)$ essentially, but
behaves better. It is deeply studied in \cite{SH}, and will play a
crucial role in our proof. We list some properties of $m(x,V)$ here,
and their proof can be found in \cite{SH}.

\begin{lemma}[Lemma 1.4, \cite{SH}]\label{lA}
Assume $V \in B_q$ for some $q>n/2$, then
there exist $C>0,\ c>0,\ k_{0}>0$, such that, for any $x,\ y$ in $\mathbb{R}^{n}$, and $0<r<R<\infty$,\\
(a) $0<m(x,V)<\infty$, \\
(b) If $h=\frac{1}{m(x,V)}$, then
$\frac{1}{h^{n-2}}\int_{B(x,h)}V(y)dy=1$,\\
(c) $ m(x,V)\sim m(y,V)$, if $ |x-y|\leq \frac{C}{m(x,V)}$, \\
(d) $ m(y,V)\leq C{\{1+|x-y|m(x,V)\}}^{k_{0}}m(x,V)$,\\
(e) $ m(y,V)\geq cm(x,V)\{1+|x-y|m(x,V)\}^{-k_{0}/(1+k_{0})}$, \\
(f) $c\fbrk{1+\abs{x-y}m(y,V)}^{1/(k_0+1)}\leq
1+\abs{x-y}m(x,V)\leq C\fbrk{1+\abs{x-y}m(y,V)}^{k_0+1}$,\\
(g) $\rev{r^{n-2}}\int_{B(x,r)}V(y)dy\leq
C\brk{\frac{R}{r}}^{(n/q)-2}\cdot \rev{R^{n-2}}\int_{B(x,R)}V(y)dy.$
\end{lemma}

Estimating the kernels mainly relies on functional calculus and a
pointwise estimate of $\Gamma(x,y,\tau)$ that was given in
\cite{SH}.

\begin{theorem}[Theorem 2.7, \cite{SH}]\label{tD}
Suppose $V \in B_{n/2}$. Then, for any $x,y\ \in \mathbb{R}^n,\ \tau
\in \mathbb{R}$, and integer $ k>0$,
$$\Gamma(x,y,\tau)\leq \frac{C_k}{\fbrk{1+\abs{\tau}^{1/2}\abs{x-y}}^k\fbrk{1+m(x,V)\abs{x-y}}^k}\cdot\rev{\abs{x-y}^{n-2}}.$$
where $C_k$ is a constant independent of $x,y,\tau$.
\end{theorem}

The next lemma is used to control the integration of $V$ on a ball.

\begin{lemma}\label{l2} Suppose $V \in B_q$ for some $q>n/2$. Let $N>
\log_2C_0+1$, where $C_0$ is the constant in (\ref{e2}). Then for
any $x_0 \in \mathbb{R}^n$, $R>0$,
$$\rev{\fbrk{1+m(x_0,V)R}^N}\int_{B(x_0,R)}
V(\xi)d\xi \leq CR^{n-2}.$$
\end{lemma}
\begin{proof}
There exists a integer $j_0\in \mathbb{Z}$ such that
$2^{j_0}R\leq\rev{m(x_0,V)}<2^{j_0+1}R$. We will discuss in
following two cases. \\
Case 1: $j_0<0$. By (\ref{e2}), Lemma \ref{l2}, and (b) of Lemma
\ref{lA}, we can get
\begin{eqnarray*}
&&\rev{\fbrk{1+m(x_0,V)R}^N}\int_{B(x_0,R)} V(\xi)d\xi\\
&\leq&\rev{(2^{-j_0})^N} \int_{B(x_0,R)}V(\xi)d\xi\\
&\leq&\rev{\fbrk{2^{-j_0}}^N}C_0^{-j_0}(2^{j_0}R)^{n-2}\\
&\leq&R^{n-2} \qquad (since\ N >\log_2C_0).
\end{eqnarray*}
Case 2: $j_0\geq 0$. By (b) and (g) of Lemma \ref{lA}, we can get
\begin{eqnarray*}
&&\rev{\fbrk{1+m(x_0,V)R}^N}\int_{B(x_0,R)} V(\xi)d\xi\\
&\leq& \int_{B(x_0,R)}V(\xi)d\xi\\
&\leq&R^{n-2} \rev{R^{n-2}}\int_{B(x_0,R)}V(\xi)d\xi\\
&\leq&R^{n-2}.
\end{eqnarray*}
This completes the proof of Lemma \ref{l2}.
\end{proof}

Before giving the estimate of the kernels, we still needs one lemma,
which is proved in \cite{SH}.

\begin{lemma}[Lemma 4.6, \cite{SH}]\label{lB}
Suppose $V \in B_{q_0}$, $q_0>1$. Assume that $-\Delta
u+(V(x)+i\tau)u=0$ in $B(x_0,2R)$ for some $x_0 \in \mathbb{R}^n,\ R>0$. Then\\
(a) for $x \in B(x_0,R)$,$$\abs{\nabla u(x)}\leq
C\sup_{B(x_0,2R)}\abs{u}\cdot
\int_{B(x_0,2R)}\frac{V(y)}{\abs{x-y}^{n-1}}dy+\frac{C}{R^{n+1}}\int_{B(x_0,2R)}\abs{u(y)}dy,$$
(b) if $(n/2)<q_0<n$, let $(1/t)=(1/q_{0})-(1/n)$,
$k_0>\log_2{C_0}+1$
$$\brk{\int_{B(x_0,R)}\abs{\nabla u}^tdx}^{1/t}\leq CR^{(n/{q_0})-2}\fbrk{1+Rm(x_0,V)}^{k_0}\sup_{B(x_0,2R)}\abs{u}.$$
\end{lemma}

Now we are ready to give the estimate of the kernels.

\begin{lemma}\label{l4}
Suppose $V\in B_q$ for some $q>n/2$. Then, there exists $\delta>0$
and for any integer $k>0$, $0<h<\abs{x-y}/16$,
\begin{eqnarray}
\abs{K_1(x,y)}&\leq& \frac{C_k}{\fbrk{1+m(x,V)\abs{x-y}}^k}\cdot
\rev{\abs{x-y}^{n-2}}V(y), \label{e7}\\
\abs{K_1(x+h,y)-K_1(x,y)} &\leq&
\frac{C_k}{\fbrk{1+m(x,V)\abs{x-y}}^k}\cdot
\frac{\abs{h}^{\delta}}{\abs{x-y}^{n-2+\delta}}V(y).\label{e8}
\end{eqnarray}
\end{lemma}

\begin{lemma}\label{l5}
Suppose $V\in B_q$ for some $q>n/2$. Then, there exists $\delta>0$
and for any integer $k>0$, $0<h<\abs{x-y}/16$,
\begin{eqnarray}
\abs{K_2(x,y)}&\leq& \frac{C_k}{\fbrk{1+m(x,V)\abs{x-y}}^k}\cdot
\rev{\abs{x-y}^{n-1}}V(y)^{1/2}, \label{e9}\\
\abs{K_2(x+h,y)-K_2(x,y)}&\leq&
\frac{C_k}{\fbrk{1+m(y,V)\abs{x-y}}^k}\cdot
\frac{\abs{h}^{\delta}}{\abs{x-y}^{n-1+\delta}}V(y)^{1/2}.\label{e10}
\end{eqnarray}
\end{lemma}

\begin{lemma}\label{l6}
Suppose $V\in B_q$ for some $n/2<q<n$. Then, there exists $\delta>0$
and for any integer $k>0$, $0<h<\abs{x-y}/16$,
\begin{eqnarray}
&&\abs{K_3(x,y)} \nonumber \\
&\leq& \frac{C_k}{\fbrk{1+m(x,V)\abs{x-y}}^k}
\frac{1}{\abs{x-y}^{n-1}}\cdot \brk{\int_{B(y,\abs{x-y})}\frac{V(\xi)}{\abs{y-\xi}^{n-1}}d\xi+\rev{\abs{x-y}}}\label{e11},\\
&&\abs{K_3(x+h,y)-K_3(x,y)} \nonumber \\
&\leq& \frac{C_k}{\fbrk{1+m(x,V)\abs{x-y}}^k}
\frac{\abs{h}^{\delta}}{\abs{x-y}^{n-1+\delta}}\cdot
\brk{\int_{B(y,\abs{x-y})}\frac{V(\xi)}{\abs{y-\xi}^{n-1}}d\xi+\rev{\abs{x-y}}}.\label{e12}
\end{eqnarray}
\end{lemma}

\begin{remark} If $V\in B_n$, then \eqref{e4} follows immediately
from \eqref{e12}. This can tell us how the kernel behave when $V$
changes. However, we don't have similar result about the smoothness
with respect to the second variable.
\end{remark}

\begin{proof}[Proof of Lemma \ref{l4}]
We easily know that $K_1(x,y)=\Gamma(x,y,0)V(y)$. It immediately
follows from Theorem \ref{tD} that, for any $x,\ y\in \mathbb{R}^n$,
\begin{eqnarray*}
\abs{K_1(x,y)}\leq \frac{C_k}{\fbrk{1+m(x,V)\abs{x-y}}^k}\cdot
\rev{\abs{x-y}^{n-2}}V(y).
\end{eqnarray*}

For (\ref{e8}), fix $x,\ y\in \mathbb{R}^n$, and fix
$n/2<q_0<\min(n,q)$, then we know $V\in B_{q_0}$. Let
$R=\frac{\abs{x-y}}{8}$, $1/t=1/{q_0}-1/n$, then $\delta=1-n/t>0$
and for any $0<h<\frac{R}{2}$, it follows from the embedding theorem
of Morrey (see \cite{GT}) that
\begin{eqnarray*}
&&\abs{K_1(x+h,y)-K_1(x,y)} \\
&\leq&
\abs{\Gamma(x+h,y,0)-\Gamma(x,y,0)}V(y)\\
&\leq&
C\abs{h}^{1-(n/t)}\brk{\int_{B(x,R)}\abs{\nabla_x\Gamma(z,y,0)}^tdz}^{1/t}V(y).
\end{eqnarray*}
and then using Lemma \ref{lB} we have
\begin{eqnarray*}
&&\abs{K_1(x+h,y)-K_1(x,y)} \\
&\leq&C\abs{h}^{1-(n/t)}R^{(n/{q_0})-2}\fbrk{1+Rm(x,V)}^{k_0}\sup_{z \in B(x,2R)}\abs{\Gamma(z,y,0)}V(y)\\
&\leq&C\frac{\abs{h}^{\delta}}{R^{\delta}}\fbrk{1+Rm(x,V)}^{k_0}\sup_{z \in B(x,2R)}\abs{\Gamma(z,y,0)}V(y)\\
&\leq&C\frac{\abs{h}^{\delta}}{R^{\delta}}\fbrk{1+Rm(x,V)}^{k_0}\sup_{z \in B(x,2R)}\frac{C_{k_1}}{\fbrk{1+m(y,V)\abs{z-y}}^{k_1}}\cdot\rev{\abs{z-y}^{n-2}}V(y)\\
&\leq&C_k\frac{\abs{h}^{\delta}}{\abs{x-y}^{\delta}}\rev{\fbrk{1+m(x,V)\abs{x-y}}^k}\cdot\rev{\abs{x-y}^{n-2}}V(y)\quad (k_1\ large).\\
\end{eqnarray*}
where we used (f) of Lemma \ref{lA} in the last inequality.
\end{proof}

\begin{proof}[Proof of Lemma \ref{l5}]
By functional calculus, we may write
$$(-\Delta+V)^{-1/2}=-\rev{2\pi}\int_\mathbb{R}(-i\tau)^{-1/2}(-\Delta+V+i\tau)^{-1}d\tau,$$
then we know that
\begin{eqnarray}\label{e13}
K_2(x,y)=-\rev{2\pi}\int_\mathbb{R}(-i\tau)^{-1/2}\Gamma(x,y,\tau)d\tau
V(y)^{1/2}.
\end{eqnarray}
In order to estimate the integration, we claim that: For $k>2$, then
\begin{eqnarray}\label{e26}
\int_\mathbb{R} \abs{\tau}^{-1/2}
\fbrk{1+\abs{\tau}^{1/2}\abs{x-y}}^{-k} d \tau \leq
\frac{C_k}{\abs{x-y}}.
\end{eqnarray}
In fact, we have
\begin{eqnarray*}
&&\int_\mathbb{R} \abs{\tau}^{-1/2}
\fbrk{1+\abs{\tau}^{1/2}\abs{x-y}}^{-k} d \tau\\
&=&\brk{\int_{\abs{\tau}\leq \abs{x-y}^{-2}}+\int_{\abs{\tau}\geq
\abs{x-y}^{-2}}} \abs{\tau}^{-1/2}
\fbrk{1+\abs{\tau}^{1/2}\abs{x-y}}^{-k} d \tau\\
&\leq&\int_{\abs{\tau}\leq\abs{x-y}^{-2}}\abs{\tau}^{-1/2}d
\tau+\int_{\abs{\tau}\geq \abs{x-y}^{-2}}\abs{\tau}^{(-k-1)/2}\abs{x-y}^{-k} d \tau\\
&\leq&\frac{C_k}{\abs{x-y}}.\\
\end{eqnarray*}

From Theorem \ref{tD} and the estimate $(\ref{e26})$, we immediately
get (\ref{e9}). For (\ref{e10}), fix $x,\ y\in \mathbb{R}^n$, and
fix $n/2<q_0<\min(n,q)$, then we know $V\in B_{q_0}$. Let
$R=\frac{\abs{x-y}}{8}$, $1/t=1/{q_0}-1/n$, then $\delta=1-n/t>0$
and for any $0<h<\frac{R}{2}$, we have
\begin{eqnarray}\label{e14}
\qquad \abs{K_2(x+h,y)-K_2(x,y)}\leq
\rev{2\pi}\int_\mathbb{R}\abs{\tau}^{-1/2}\abs{\Gamma(x+h,y,\tau)-\Gamma(x,y,\tau)}d\tau
V(y)^{1/2}.
\end{eqnarray}

Similarly, it follows from the embedding theorem of Morrey and Lemma
\ref{lB} that
\begin{eqnarray*}
&&\abs{\Gamma(x+h,y,\tau)-\Gamma(x,y,\tau)}\\
&\leq&C\abs{h}^{1-(n/t)}\brk{\int_{B(x,R)}\abs{\nabla_x\Gamma(z,y,\tau)}^tdz}^{1/t}\\
&\leq&C\abs{h}^{1-(n/t)}R^{(n/{q_0})-2}\fbrk{1+Rm(x,V)}^{k_0}\sup_{z \in B(x,2R)}\abs{\Gamma(z,y,\tau)}\\
&\leq&C\frac{\abs{h}^{\delta}}{R^{\delta}}\fbrk{1+Rm(x,V)}^{k_0}\sup_{z \in B(x,2R)}\abs{\Gamma(z,y,\tau)}\\
&\leq&C\frac{\abs{h}^{\delta}}{R^{\delta}}\fbrk{1+Rm(x,V)}^{k_0}\sup_{z \in B(x,2R)}\frac{C_k\fbrk{1+\abs{\tau}^{1/2}\abs{z-y}}^{-k}}{\fbrk{1+m(y,V)\abs{z-y}}^k}\cdot\rev{\abs{z-y}^{n-2}}\\
&\leq&C_k\frac{\abs{h}^{\delta}}{\abs{x-y}^{\delta}}\frac{\fbrk{1+\abs{\tau}^{1/2}\abs{x-y}}^{-k}}{\fbrk{1+m(y,V)\abs{x-y}}^k}\cdot\rev{\abs{x-y}^{n-2}}.\\
\end{eqnarray*}
Hence, insert this to (\ref{e14}), it follows from the estimate
$(\ref{e26})$ that
$$\abs{K_2(x+h,y)-K_2(x,y)}\leq C_k \frac{\abs{h}^{\delta}}{\abs{x-y}^{n-1+\delta}}\frac{1}{\fbrk{1+m(y,V)\abs{x-y}}^k}V(y)^{1/2}.$$
\end{proof}

\begin{proof}[Proof of Lemma \ref{l6}]
By partial integral, we know that
\begin{eqnarray}\label{e15}
K_3(x,y)=\rev{2\pi}\int_\mathbb{R}(-i\tau)^{-1/2}\nabla_y\Gamma(x,y,\tau)d\tau.
\end{eqnarray}
Fix $x,\ y\in \mathbb{R}^n$, Let $R=\frac{\abs{x-y}}{8}$,
$1/t=1/q-1/n$, $\delta=n/q-2>0$, and for any $0<h<\frac{R}{2}$, we
have
\begin{eqnarray}\label{e16}
\abs{K_3(x+h,y)-K_3(x,y)}\leq
\rev{2\pi}\int_\mathbb{R}\abs{\tau}^{-1/2}\abs{\nabla_y\Gamma(x+h,y,\tau)-\nabla_y\Gamma(x,y,\tau)}d\tau.
\end{eqnarray}
Similarly, it follows from the imbedding theorem of Morrey and Lemma
\ref{lB} that
\begin{eqnarray}
&&\abs{\nabla_y\Gamma(x+h,y,\tau)-\nabla_y\Gamma(x,y,\tau)}\label{e17} \\
&\leq&C\abs{h}^{1-(n/t)}\brk{\int_{B(x,R)}\abs{\nabla_x\nabla_y\Gamma(z,y,\tau)}^tdz}^{1/t} \nonumber \\
&\leq&C\abs{h}^{1-(n/t)}R^{(n/{q})-2}\fbrk{1+Rm(x,V)}^{k_0}\sup_{z
\in B(x,2R)}\abs{\nabla_y\Gamma(z,y,\tau)}. \nonumber
\end{eqnarray}
Since $\Gamma(z,y,\tau)=\Gamma(y,z,-\tau)$, then
$\nabla_y\Gamma(z,y,\tau)=\nabla_x\Gamma(y,z,-\tau)$. It follows
from (a) of Lemma \ref{lB} that
\begin{eqnarray*}
&&\sup_{z \in B(x,2R)}\abs{\nabla_y\Gamma(z,y,\tau)} \leq \sup_{z
\in B(x,2R)}\abs{\nabla_x\Gamma(y,z,-\tau)}\\
&\leq&\sup_{z \in B(x,2R)}\ \{\sup_{\eta \in
B(y,\abs{y-z}/4)}\abs{\Gamma(\eta,z,-\tau)}\cdot
\int_{B(y,\abs{z-y}/2)}\frac{V(\xi)}{\abs{y-\xi}^{n-1}}d\xi\nonumber\\
&&\qquad\qquad\qquad+\frac{C}{\abs{y-z}^{n+1}}\int_{B(y,\abs{z-y}/2)}\Gamma(\xi,z,-\tau)d\xi\
\}.\nonumber
\end{eqnarray*}

Using the fact that $\abs{\eta-z}\sim\abs{y-z}$,
$\abs{\xi-z}\sim\abs{y-z}$ and $\abs{x-y}\sim\abs{y-z}$, choosing
$k_1$ sufficiently large, it follows from Theorem \ref{tD} and (f)
of Lemma \ref{lA} that
\begin{eqnarray}
&&\sup_{z \in B(x,2R)}\abs{\nabla_y\Gamma(z,y,\tau)}\label{e18}\\
&\leq&\sup_{z \in
B(x,2R)}\frac{C_{k_1}}{\{1+\abs{\tau}^{1/2}\abs{y-z}\}^{k_1}
\{1+m(z,V)\abs{y-z}\}^{k_1}}\cdot
\rev{\abs{y-z}^{n-2}}\int_{B(y,\abs{x-y})}\frac{V(\xi)}{\abs{y-\xi}^{n-1}}d\xi\nonumber\\
&&\qquad\qquad\qquad+\frac{C_{k_1}}{\{1+\abs{\tau}^{1/2}\abs{y-z}\}^{k_1}
\{1+m(z,V)\abs{y-z}\}^{k_1}}\cdot \rev{\abs{y-z}^{n-1}}\nonumber\\
&\leq&\frac{C_k}{\{1+\abs{\tau}^{1/2}\abs{x-y}\}^k\{1+m(x,V)\abs{x-y}\}^k}\cdot \rev{\abs{x-y}^{n-2}}\int_{B(y,\abs{x-y})}\frac{V(\xi)}{\abs{y-\xi}^{n-1}}d\xi\nonumber\\
&&\qquad\qquad\qquad+\frac{C_k}{\{1+\abs{\tau}^{1/2}\abs{x-y}\}^k\{1+m(x,V)\abs{x-y}\}^k}\cdot
\rev{\abs{x-y}^{n-1}}.\nonumber
\end{eqnarray}
From the estimate $(\ref{e26})$ and (\ref{e18}), we immediately get
(\ref{e11}). Inserting (\ref{e18}) to (\ref{e17}), we get that
\begin{eqnarray}
&&\abs{\nabla_y\Gamma(x+h,y,\tau)-\nabla_y\Gamma(x,y,\tau)}\label{e19}\\
&\leq&C_k
\frac{\abs{h}^{\delta}}{\abs{x-y}^{\delta}}\frac{C_k}{\{1+\abs{\tau}^{1/2}\abs{x-y}\}^k\fbrk{1+m(x,V)\abs{x-y}}^k}\nonumber\\
&&\cdot
\brk{\rev{\abs{x-y}^{n-2}}\int_{B(y,\abs{x-y})}\frac{V(\xi)}{\abs{y-\xi}^{n-1}}d\xi+\rev{\abs{x-y}^{n-1}}}.\nonumber
\end{eqnarray}
Inserting (\ref{e19}) to (\ref{e16}),  we get from the estimate
$(\ref{e26})$ that
\begin{eqnarray*}
&&\abs{K_3(x+h,y)-K_3(x,y)}\\
&\leq& C_k
\frac{\abs{h}^{\delta}}{\abs{x-y}^{\delta}}\frac{1}{\fbrk{1+m(x,V)\abs{x-y}}^k}\cdot\brk{\rev{\abs{x-y}^{n-1}}\int_{B(y,\abs{x-y})}\frac{V(\xi)}{\abs{y-\xi}^{n-1}}d\xi+\rev{\abs{x-y}^{n}}}.\\
\end{eqnarray*}
\end{proof}

\section{Proof of main results}

We first discuss the problem for general operator $Tf(x)=\int
K(x,y)f(y)dy$. Later, we will specialize to $T_j\ (j=1,2,3)$.
\begin{proposition}\label{p1}
Let $m>1$, suppose $T$ is bounded on $L^p$ for every $p \in
(m',\infty)$, and $K$ satisfies $H(m)$, then $\forall\ b\ \in
\mathbf{BMO}$, $[b,T]$ is bounded on $L^p$ for every $p \in
(m',\infty)$, and $$\normo{[b,T]f}_p \leq C_p
\normo{b}_{\mathbf{BMO}} \normo{f}_p.$$
\end{proposition}

We adopt the idea of Str\"omberg (see \cite{TO}). Recall that the
sharp function of Fefferman-Stein is defined by
\begin{eqnarray}\label{e20}
M^{\sharp}f(x)=\sup_{x\in B}\rev{|B|} \int_B\abs{f(y)-f_B}dy,
\end{eqnarray}
where $f_B=\rev{|B|}\int_Bf(y)dy$, and the supremum is taken on all
balls $B$ with $x \in B$.

Recall that $\mathbf{BMO}$ is defined by
\begin{eqnarray}\label{e21}
\mathbf{BMO}(\mathbb{R}^n)=\{f \in
L_{loc}^1(\mathbb{R}^n):\normo{f}_{\mathbf{BMO}}=\normo{M^{\sharp}f}_{\infty}<\infty
\}.
\end{eqnarray}
Two basic facts about $\mathbf{BMO}$ may be in order. We use $2^kB$
to denote the ball with the same center as $B$ but with $2^k$ times
radius.
\begin{eqnarray}\label{e22}
\abs{f_{2^kB}-f_B}\leq C (k+1)\normo{f}_{\mathbf{BMO}},\mbox{ for }
k>0.
\end{eqnarray}
The second one is due to John-Nirenberg.
\begin{eqnarray}\label{e23}
\normo{f}_{\mathbf{BMO}} \sim \sup_{B}
\brk{\rev{|B|}\int_B\abs{f(y)-f_B}^pdy}^{1/p},\mbox{ for any } p>1.
\end{eqnarray}
Proposition \ref{p1} follows immediately from the following lemma
and a theorem of Fefferman-Stein on sharp function.
\begin{lemma}\label{l8}
Let $T$ satisfies the same condition in Proposition \ref{p1}. Then
$\forall s>m',$ there exists constant $C_s>0, $ such that $ \forall\
f \in L_{loc}^1,\ b\ \in \mathbf{BMO}$
\begin{eqnarray}\label{e24}
 M^\sharp([b,T]f)(x)\leq
C_s\normo{b}_{\mathbf{BMO}}\{M_s(Tf)(x)+M_s(f)(x)\},
\end{eqnarray}
where $M_s(f)=M(|f|^s)^{1/s}$ and $M$ is Hardy-Littlewood maximal
function.
\end{lemma}

%=====================  1  ==========================%
\begin{proof}
Fix $s>m',\ f \in L_{loc}^1, \  x \in \mathbb{R}^n$, and fix
a ball $I=B(x_0,l)$ with $ x \in I.$ We only need to control
$J=\frac{1}{\abs I}\int_I \abs{[b,T]f(y)-([b,T]f)_I} dy$ by the
right side of (\ref{e24}). Let $f=f_1+f_2$, where $f_1=f\chi_{32I},\
f_2=f-f_1$. Then
$[b,T]f=[b-b_I,T]f=(b-b_I)Tf-T(b-b_I)f_1-T(b-b_I)f_2\triangleq
A_1f+A_2f+A_3f$,  and we get
\begin{eqnarray*}
J & \leq & \frac{1}{\abs I}\int_I \abs {A_1f(y)-(A_1f)_I} dy\\&
&+\frac{1}{\abs I}\int_I \abs {A_2f(y)-(A_2f)_I} dy+\frac{1}{\abs
I}\int_I \abs{ A_3f(y)-(A_3f)_I} dy \\
& \triangleq & J_1+J_2+J_3.
\end{eqnarray*}
Step 1. First we consider $J_1$. By H\"older inequality and
(\ref{e23}),
%===============  J1  ===================%
\begin{eqnarray*}
 J_1 &\leq& \frac{2}{\abs I}\int_I \abs {A_1f(y)} dy\\
&=&\frac{2}{\abs I}\int_I \abs {(b-b_I)Tf(y)} dy\\
&\leq& 2 \brk{\frac{1}{\abs I}\int_I \abs {(b-b_I)} ^{s'}
dy}^{\rev{s'}}\brk{\frac{1}{\abs I}\int_I \abs{Tf(y)} ^s dy}^{\rev s} \\
&\leq& 2 \normo{b}_{\mathbf{BMO}} M_s(Tf)(x).
\end{eqnarray*}
Step 2. Second we consider $J_2$. Fix $s_1$ such that $s>s_1>m'$,
and let $s_2=\frac{ss_1}{s-s_1}$, then we have
\begin{eqnarray*}
 J_2 &\leq&2 \rev{\abs I}\int_I \abs {A_2f(y)} dy\\
     &\leq&2\brk{\rev{\abs I}\int_I \abs {A_2f(y)} ^{s_1} dy}^{\rev {s_1}}\\
     &\leq&2\brk{\rev{\abs I}\int_{32I} \abs {(b-b_I)f(y)} ^{s_1} dy}^{\rev {s_1}}\\
     &\leq&C\brk{\rev{\abs{32I}}\int_{32I} \abs {b-b_I} ^{s_2} dy}^{\rev {s_2}}\brk{\rev{\abs{32I}}\int_{32I} \abs {f(y)} ^s dy}^{\rev s} \\
&\leq& C \normo{b}_{\mathbf{BMO}} M_s(f)(x).
\end{eqnarray*}
Step 3. Last we consider $J_3$. Set $c_I=\int_{\abs{z-x_0} >32l}
K(x_0,z)(b(z)-b_I)f(z)dz$, then we have that
\begin{eqnarray*}
J_3 &\leq& \frac{2}{\abs I}\int_I \abs {A_3f(y)-c_I} dy.\\
 &\leq& 2\rev{\abs I}\int_I \aabs {\int_{\abs{z-x_0} \geq 32l} \fbrk{K(y,z)-K(x_0,z)}(b(z)-b_I)f(z)dz} dy\\
&\leq&2\rev{\abs I}\int_I \int_{\abs{z-x_0} > 32l}\abs{\fbrk{K(y,z)-K(x_0,z)}(b(z)-b_I)f(z)}dz dy\\
&=&2\rev{\abs I}\int_I \sum_{k=5}^{\infty} \int_{2^kl \leq \abs{z-x_0} < 2^{k+1}l}\abs{\fbrk{K(y,z)-K(x_0,z)}(b(z)-b_I)f(z)}dz dy.\\
\end{eqnarray*}
From H\"older's inequality, we get

\begin{eqnarray*}
J_3&\leq&2\rev{\abs I}\int_I\ \sum_{k=5}^{\infty} \brk{\int_{2^kl \leq \abs{z-x_0} < 2^{k+1}l}\abs{K(y,z)-K(x_0,z)}^mdz}^{1/m}\\
& &\qquad\cdot\brk{\int_{2^kl \leq \abs{z-x_0} < 2^{k+1}l}\abs{(b(z)-b_I)f(z)}^{m'}dz}^{1/m'}\ dy\\
&\leq&2\rev{\abs I}\int_I\ \sum_{k=5}^{\infty}\brk{ \int_{2^kl \leq \abs{z-x_0} < 2^{k+1}l}\abs{K(y,z)-K(x_0,z)}^mdz}^{1/m}(2^kl)^{n/m'}k\\
& &\cdot \rev{(2^kl)^{n/m'}k}\brk{\int_{2^kl \leq \abs{z-x_0} < 2^{k+1}l}\abs{(b(z)-b_I)f(z)}^{m'}dz}^{1/m'}\ dy\\
&\leq&C \sup_{k \geq 5}\rev{(2^kl)^{n/m'}k}\brk{\int_{2^kl \leq
\abs{z-x_0} < 2^{k+1}l}\abs{(b(z)-b_I)f(z)}^{m'}dz}^{1/m'}\\
&\leq&C \sup_{k \geq5}\rev{k}\brk{\rev{(2^kl)^{n}}\int_{\abs{z-x_0}
<2^{k+1}l}\abs{(b(z)-b_{2^{k+1}I}+b_{2^{k+1}I}-b_I)f(z)}^{m'}dz}^{1/m'}\\
&\leq&C \sup_{k\geq5}\rev{k}{(k+2)\normo{b}_{\mathbf{BMO}}M_sf(x)}\quad (by\ (\ref{e22}))\\
 &\leq&C \normo{b}_{\mathbf{BMO}}M_sf(x) .
\end{eqnarray*}
This completes the proof of lemma \ref{l8}.
\end{proof}

\begin{proof}[Proof of Theorem \ref{t1}]
Now we begin to prove Theorem \ref{t1}. Considering Remark \ref{r1},
we can assume $q>\frac{n}{2}, q'<p$. We first prove (i). By
Proposition \ref{p1} and Theorem \ref{tA}, it suffices to prove that
$K_1$ satisfies $H(q)$ (see (\ref{e6})). From (\ref{e8}), we have
\begin{eqnarray*}
&&\brk{\int_{2^kl \leq \abs{y-x_0} < 2^{k+1}l}
\abs{K_1(x,y)-K_1(x_0,y)}^{q}dy}^{1/q}\\
&\leq&C_N\frac{l^{\delta}}{(2^kl)^{n-2+\delta}}\rev{\fbrk{1+m(x_0,V)2^kl}^N}\int_{B(x_0,2^{k+3}l)}
V(y)^q dy^{1/q}\\
&\leq&C_N\frac{l^{\delta}}{(2^kl)^{n-2+\delta}}\rev{\fbrk{1+m(x_0,V)2^kl}^N}(2^{k}l)^{-n/q'}\int_{B(x_0,2^{k}l)}
V(\xi)d\xi\\
&\leq&C_N\frac{l^{\delta}}{(2^kl)^{n-2+\delta}}(2^{k}l)^{n/q-2}\quad (by\ lemma\ \ref{l2})\\
&\leq&C\frac{l^{\delta}}{(2^kl)^{(n/{q'})+\delta}}.
\end{eqnarray*}
Thus, we can get
\begin{eqnarray*}
&&\sum_{k=5}^{\infty} k(2^kl)^{\frac{n}{q'}}\brk{\int_{2^kl \leq
\abs{y-x_0} < 2^{k+1}l}
\abs{K_1(x,y)-K_1(x_0,y)}^{q}dy}^{1/q}\\
&\leq&\sum_{k=5}^{\infty} \frac{Ck}{(2^k)^{\delta}}\leq C.
\end{eqnarray*}

For the proof of (ii).--- It suffices to prove that $K_2$ satisfies
$H(2q)$. From (\ref{e10}), we have
\begin{eqnarray*}
&&\brk{\int_{2^kl \leq \abs{y-x_0} < 2^{k+1}l}
\abs{K_2(x,y)-K_2(x_0,y)}^{2q}dy}^{1/(2q)}\\
&\leq&C_N\frac{l^{\delta}}{(2^kl)^{n-1+\delta}}\rev{\fbrk{1+m(x_0,V)2^kl}^N}\int_{B(x_0,2^{k+3}l)}
V(\xi)^{q}d\xi^{1/(2q)}\\
&\leq&C_N\frac{l^{\delta}}{(2^kl)^{n-1+\delta}}\rev{\fbrk{1+m(x_0,V)2^kl}^N}(2^{k}l)^{-n/(2q')}\int_{B(x_0,2^{k}l)}
V(\xi)d\xi^{1/2}\\
&\leq&C_N\frac{l^{\delta}}{(2^kl)^{n-1+\delta}}(2^{k}l)^{-n/(2q')+(n-2)/2}\\
&\leq&C\frac{l^{\delta}}{(2^kl)^{\delta}}(2^kl)^{-n/{(2q)'}},
\end{eqnarray*}
hence, we get
\begin{eqnarray*}
&&\sum_{k=5}^{\infty} k(2^kl)^{\frac{n}{(2q)'}}\brk{\int_{2^kl \leq
\abs{y-x_0} < 2^{k+1}l}
\abs{K_2(x,y)-K_2(x_0,y)}^{2q}dy}^{1/(2q)}\\
&\leq&\sum_{k=5}^{\infty} \frac{Ck}{(2^k)^{\delta}}\leq C.
\end{eqnarray*}

Last, we prove (iii).--- It suffices to prove that $K_3$ satisfies
$H(p_0)$. From (\ref{e12}), we have
\begin{eqnarray*}
&&\brk{\int_{2^kl \leq \abs{y-x_0} < 2^{k+1}l}
\abs{K_3(x,y)-K_3(x_0,y)}^{p_0}dy}^{1/p_0}\\
&\leq&C_N\frac{l^{\delta}}{(2^kl)^{n-1+\delta}}\rev{\fbrk{1+m(x_0,V)2^kl}^N}\normo{\int
\frac{V(\xi)\chi_{B(x_0,2^{k+3}l)}}{\abs{y-\xi}^{n-1}}d\xi}_{L_y^{p_0}}+\frac{l^{\delta}}{(2^kl)^{(n/{p_0'})+\delta}}\\
&\leq&C_N\frac{l^{\delta}}{(2^kl)^{n-1+\delta}}\rev{\fbrk{1+m(x_0,V)2^kl}^N}\int_{B(x_0,2^{k+3}l)}
V(\xi)^{q}d\xi^{1/q}+\frac{l^{\delta}}{(2^kl)^{(n/{p_0'})+\delta}}\\
&\leq&C_N\frac{l^{\delta}}{(2^kl)^{n-1+\delta}}\rev{\fbrk{1+m(x_0,V)2^kl}^N}(2^{k}l)^{-n/q'}\int_{B(x_0,2^{k}l)}
V(\xi)d\xi+\frac{l^{\delta}}{(2^kl)^{(n/{p_0'})+\delta}}\\
&\leq&C_N\frac{l^{\delta}}{(2^kl)^{n-1+\delta}}(2^{k}l)^{n/q-2}+\frac{l^{\delta}}{(2^kl)^{(n/{p_0'})+\delta}}\\
&\leq&C\frac{l^{\delta}}{(2^kl)^{(n/{p_0'})+\delta}},
\end{eqnarray*}
therefore, we get
\begin{eqnarray*}
&&\sum_{k=5}^{\infty} k(2^kl)^{\frac{n}{p_0'}}\brk{\int_{2^kl \leq
\abs{y-x_0} < 2^{k+1}l}
\abs{K_3(x,y)-K_3(x_0,y)}^{p_0}dy}^{1/p_0}\\
&\leq&\sum_{k=5}^{\infty} \frac{Ck}{(2^k)^{\delta}}\leq C.
\end{eqnarray*}
\end{proof}

\section{The Converse Result}

This section is devoted to the converse problem. Recall that
$T_3=\nabla(-\Delta+V)^{-1/2}$ is the Riesz transform associated to
Schr\"odinger operator. A natural problem is that whether the
converse holds. Namely, if $[b,T_3]$ is bounded on $L^2$, do we have
$b \in \mathbf{BMO}$? This is quite subtle. If $V\equiv0$, it
reduces to the classical Riesz transform. However, for general $V\in
B_q$, the converse fails. Considering $V\equiv 1$, which is in $B_q$
for every $q>1$, we have the following,
\begin{theorem}
There exist a function $b \notin \mathbf{BMO}$, such that $[b,T_3]$
is bounded on $L^2$.
\end{theorem}
\begin{proof}
Consider $b=x_j$, we know that $b \notin \mathbf{BMO}$. We have
that,
$$[b,T_3]f=x_j\nabla(-\Delta+1)^{-1/2}f-\nabla(-\Delta+1)^{-1/2}(x_jf).$$
From Plancherel equality, we can get
\begin{eqnarray*}
\normo{[b,T_3]f}_2&=&\normo{\partial_j\brk{\frac{\xi}{(1+\xi^2)^{1/2}}
\hat{f}}-\frac{\xi}{(1+\xi^2)^{1/2}} \partial_j\hat{f}}_2\\
&=&\normo{\partial_j\brk{\frac{\xi}{(1+\xi^2)^{1/2}}} \hat{f}}_2\\
&\leq&\normo{f}_2.
\end{eqnarray*}
\end{proof}

The converse example in Theorem 2 implies that the assumption $V\in
B_{q}$ is too weak, it can not guarantee the function $b\in BMO$.
However if we assume $V$ satisfies some additional conditions, for
example, if $V$ is $L^p$ integrable, then the converse could be
true. Let $T_3'=(-\Delta)^{1/2}(-\Delta+V)^{-1/2}$, then from
$T_{3}^{'}=(-\triangle)^{-1/2}\nabla\cdot T_{3}$, we know the
results above also hold with $T_3$ replaced by $T_3'$.
\begin{theorem}
If $[b,T_3]$, $[b,T_3']$ and $V^{1/2}(-\triangle)^{-1/2}$ is bounded
on $L^2$, then $b\in \mathbf{BMO}$.
\end{theorem}
\begin{proof}
From $[b,T_3]$, $[b,T_3']$ is bounded on $L^2$, and
$$[b,T_3]=[b,\nabla(-\Delta)^{-1/2}T_3']=[b,\nabla(-\Delta)^{-1/2}]T_3'+\nabla(-\Delta)^{-1/2}[b,T_3'],$$
we have $[b,\nabla(-\Delta)^{-1/2}]T_3'$ is bounded on $L^2$.

We claim that, $[b,\nabla(-\Delta)^{-1/2}]$ is bounded on $L^2$,
which implies the theorem from the well known theorem of Coifman,
Rochberg and Weiss. It suffices to prove that $T_3'$ has a converse
bounded on $L^2$. Note that
$T_3'^{-1}=(-\Delta+V)^{1/2}(-\Delta)^{-1/2}$, and
\begin{eqnarray*}
T_3'^{-1}f&=&(-\Delta+V)^{1/2}(-\Delta)^{-1/2}f\\
&=&(-\Delta+V)^{-1/2}(-\Delta+V)(-\Delta)^{-1/2}f\\
&=&(-\Delta+V)^{-1/2}(-\Delta)^{1/2}f+(-\Delta+V)^{-1/2}V^{1/2}V^{1/2}(-\Delta)^{-1/2}f,
\end{eqnarray*}

Therefore, by using $V^{1/2}(-\triangle)^{-1/2}$ is bounded on
$L^{2}$, we can easily get the conclusion of Theorem 3.
\end{proof}

\begin{corollary}
If $[b,T_3]$, $[b,T_3']$ is bounded on $L^2$, and $V\in
L^{n/2}\bigcap B_q$ for $q>n/2$, then $b\in \mathbf{BMO}$.
\end{corollary}
\begin{proof}
we only need to prove that $V^{1/2}(-\Delta)^{-1/2}$ is bounded on
$L^2$. This follows directly from H\"older inequality and fractional
integration that,
$$\normo{V^{1/2}(-\Delta)^{-1/2}f}_2\leq C\normo{V^{1/2}}_{n}\normo{(-\Delta)^{-1/2}f}_{2n/(n-2)}\leq
C\normo{V}_{n/2}^{1/2}\normo{f}_2.$$
\end{proof}

%\end{multicols}

\end{document}